\documentclass{article}

\usepackage{hyperref}
\usepackage[english]{babel}

\usepackage{graphicx}
\usepackage{amssymb}
\usepackage{amsthm, amsmath,amssymb}
\usepackage[all]{xy}
\usepackage{graphicx}
\usepackage{epstopdf}
\usepackage{graphicx}
\usepackage{psfrag}
\usepackage{multirow}
\usepackage{tikz}
\usepackage{color}

\usepackage[applemac]{inputenc}

\def\Rb{\mathbb{R}}		%reals
\def\Cb{\mathbb{C}}          %complex
\newcommand{\Zb}{\mathbb{Z}}
\newcommand{\Nb}{\mathbb{N}}

\usepackage{tikz}
\usetikzlibrary{automata}

\newcommand{\restr}[1]
   {\vrule height1ex width.4pt
depth1.4ex\lower1.4ex\hbox{\scriptsize $\,#1$}}

\newtheorem{theorem}{Theorem}
\newtheorem{corollary}[theorem]{Corollary}
\newtheorem{proposition}{Proposition}

\newtheorem{lemma}{Lemma}

\theoremstyle{definition}

\newtheorem{remark}{Remark}

\author{Inês Cruz\thanks{Centro de Matem\'atica da Universidade do Porto (CMUP), Departamento de Matem\'atica,  Faculdade de Ci\^encias da Universidade do Porto, R. Campo Alegre, 687, 4169-007 Porto, Portugal.E-mails: imcruz@fc.up.pt, mmmatos@fc.up.pt},\,\, Helena Mena-Matos\footnotemark[1]\,\, and M. Esmeralda Sousa-Dias\thanks{Center for Mathematical Analysis, Geometry and Dynamical Systems (CAMGSD),
Departamento de Matem\'atica, Instituto Superior T\'ecnico, Av. Rovisco Pais, 1049-001 Lisboa, Portugal. E-mail:e.sousa.dias@tecnico.ulisboa.pt}}

\begin{document}
\title{Dynamics and periodicity in a family of cluster maps}
%\title{Periodicity and dynamics of a family of cluster maps}

\maketitle

\begin{abstract}
The dynamics of a 1-parameter family of cluster maps $\varphi_r$  associated to mutation-periodic quivers in dimension 4, is studied in detail. The use of presymplectic reduction leads to a globally periodic symplectic map, and this enables us to reduce the problem to the study of maps belonging to a  group of  symplectic birational maps of the plane which is  isomorphic to $SL(2,\mathbb{Z})\ltimes\mathbb{R}^2$. We conclude that there  are three different types of dynamical behaviour for $\varphi_r$ characterized by the integer parameter values $r=1$, $r=2$ and $r>2$.  For each type, the periodic points,  the structure and the asymptotic behaviour of the orbits are completely described. A finer description of the dynamics is provided  by using first integrals.

\end{abstract}
 
\medskip

\noindent {\it MSC 2010:} 39A20, 37E15, 37J10, 37J15, 13F60.\\

\noindent {\it Keywords:}   mutation-periodicity, cluster maps, global periodicity,  symplectic maps, dynamics.

\section{Introduction}

We study the dynamics in $\Rb^4_+$ of a family of maps which arises in the context of  the theory of cluster algebras  associated to 4-node quivers with mutation-period equal to 2. This family depends on a positive integer parameter $r$ and is defined by
\begin{equation}\label{Family1}
\varphi_r(x_1,x_2,x_3,x_4)=  \left(x_3,x_4,\frac{x_2^r+x_3^r}{x_1}, \frac{x_1^rx_4^r+ (x_2^r+x_3^r)^r}{x_1^rx_2}\right).
\end{equation}

The study performed in this work provides the full description of the dynamics of the family of maps \eqref{Family1} and enables us to conclude that there are three  different types of dynamical behaviour characterised by $r=1$, $r=2$ and  $r>2$, respectively. The dynamics of $\varphi_r$ in the cases $r=1$ and $r>2$ is described  in Theorem~\ref{tr=1} and Theorem~\ref{tr>2} respectively, and for $r=2$ it can be found in \cite[Theorem 3]{InHeEs}. In particular,  in what concerns the existence of periodic points of  the map $\varphi_r$:  (a)   $\varphi_1$ is globally 12-periodic, with a unique fixed point and 2-dimensional algebraic subvarieties of points with minimal period 4 and 6; (b) $\varphi_2$ has no periodic points; (c)  if $r>2$, $\varphi_r$ has a unique fixed point and a 2-dimensional algebraic subvariety of periodic points of minimal period 4. Moreover, using first integrals, we also give a  very detailed description of the $\varphi_r$-orbits.

The maps $\varphi_r$, whose iterates define the following system of difference equations:
$$\left\{\begin{matrix}
x_{2n+3}x_{2n-1}&=x_{2n}^r+x_{2n+1}^r&\\
&&\qquad n=1,2,\ldots\\
x_{2n+4}\, x_{2n}&=x_{2n+3}^r+x_{2n+2}^r&\\
\end{matrix}\right. ,$$ 
are associated to the 4-node quivers $Q_r$ in Figure~\ref{quiver}, which have the property of being mutation-periodic quivers with period equal to 2. The notion of \emph{mutation-periodic quiver} was introduced by Fordy and Marsh in \cite{FoMa} in the context of Fomin and Zelevinsky's theory of cluster algebras \cite{FoZe}  and to such a quiver one associates a map whose iterates define a discrete dynamical system. A  map associated to a mutation-periodic quiver is called a  \emph{cluster map}.

\medskip

An important feature of any mutation-periodic quiver $Q$ with $N$ nodes represented by a  skew-symmetric matrix $B$ is the possibility of reducing the associated cluster map $\varphi: \Rb^N_+\rightarrow \Rb^N_+$ to a symplectic map $\widehat{\varphi}:  \Rb^{2k}_+\rightarrow \Rb^{2k}_+$, where $2k$ is the rank of  $B$. In fact, as proved by the first and last authors in \cite{InEs2}, there exists a  semiconjugacy $\Pi: \Rb^N_+\rightarrow \Rb^{2k}_+$ such that the following diagram is commutative
$$\xymatrix{ \Rb^N_+  \ar @{->} [r]^{\Pi}\ar @{->} [d]_{\varphi} & \Rb^{2k}_+ \ar @{->}[d]^{\widehat{\varphi}}\\
\Rb^{N}_+  \ar @{->} [r]_{\Pi} & \Rb^{2k}_+}$$
Also, Darboux type coordinates can be chosen so that the reduced map $\widehat{\varphi}$ preserves the symplectic form 
$$\omega=\sum_{1\leq i\leq k} \frac{dy_{2i-1}\wedge dy_{2i}}{y_{2i-1} y_{2i}}.$$

As the rank of the matrices $B_r$ associated to the family of maps $\varphi_r$ is equal to 2 (for any $r$), the reduced maps $\widehat{\varphi}_r$ are  maps in $\Rb_+^2$. As will be proved in Proposition~\ref{prop1} the symplectic reduced  maps  $\widehat{\varphi}_r$ are all conjugate to a (parameter independent) globally 4-periodic map $\psi$. This property of global periodicity turns out to be the key feature to the  successful description of the dynamics of the maps $\varphi_r$,  since it allows us to study these maps by restricting them and their fourth iterates to 2-dimensional varieties of $\Rb_+^4$. Moreover, these  restricted maps belong to a group of symplectic birational maps whose dynamical behaviour we are able to describe completely. The dynamics of the original family $\varphi_r$ is  then obtained from the dynamics of these restricted maps.

\medskip 

Since 2000, many  connections and applications of the theory of cluster algebras to diverse areas of Mathematics and Physics have been unveiled. Namely, applications to  integrable systems, Poisson geometry, algebraic geometry, quiver representations, Teichm\"uller theory and tropical geometry (see for instance \cite{WL}, \cite{GeLeSc}, \cite{GeShVa3}, \cite{Nak}, and \cite{Lec}). In particular,
 examples of mutation-periodic quivers   appear  in supersymmetric quiver gauge theories associated to complex cones over several surfaces such as the Hirzebruch 0 and the del Pezzo 0-3 surfaces (see \cite[ \textsection 11]{FoMa}). For instance, the quiver $Q_2$ (in Figure~\ref{quiver}) is the one associated to the Hirzebruch 0 surface \cite{FengHa}. Moreover, the rule defining the mutation at a node of a quiver, as defined in 2000 by Fomin and Zelevinsky \cite{FoZe}, coincides with the rule in supersymmetric  quiver gauge theories for Seiberg-dualising a quiver at a given node (see \cite[\textsection 3]{MuRa}).

Let us refer that in the context of cluster algebras, besides the notion of mutation-periodicity there  are other  periodicity notions  (see  \cite{Na} and \cite[\textsection 7] {Marsh}).  For instance the  so-called categorical periodicity which has been used to reformulate and prove the Zamolodchikov’s periodicity conjecture for Y-systems \cite{Za}.  Although this conjecture arose from studies  of the thermodynamic
Bethe ansatz in mathematical physics, it was  proved in the cluster algebras context  by  Fomin-Zelevinsky \cite{FoZeY} for Dynkin diagrams and by Keller \cite{Ke} for pairs of Dynkin diagrams.

There are not many works in the literature addressing the dynamics of cluster maps. In \cite{FoHo2} and \cite{FoHo1} we can find the study of   integrability and algebraic entropy of  cluster maps arising from 1-periodic quivers. Based on results in \cite{InEs2}, which show how to constructively reduce cluster maps associated to mutation-periodic quivers of arbitrary period, we were able to study the dynamics of cluster maps arising from the quivers associated to the Hirzebruch 0 and del Pezzo 3 surfaces (see \cite{InHeEs}). That work and the study presented here are, to the best of our knowledge, the only ones approaching the dynamics of cluster maps associated to higher periodic quivers.

\medskip 

The structure of the paper is as follows.  In Section~\ref{reduction} we compute the reduced symplectic maps $\widehat\varphi_r$ and show that they are all conjugate to a globally 4-periodic map $\psi$. As a consequence, the dynamics of the maps $\varphi_r$ reduces to the  study of  the restriction of $\varphi_r$ and of its fourth iterate, $\varphi_r^{(4)}$, to certain $2$-dimensional varieties (Proposition~\ref{orbits}). All these restricted maps are then shown to belong to a  specific group $\Gamma$ of birational maps. 

Section~\ref{gamma} is devoted to the study of the maps belonging to the group $\Gamma$ which, we show, is isomorphic to the semidirect product  $SL(2,\Zb)  \ltimes \Rb^2$.  Firstly we find normal forms  for  maps of $\Gamma$ up to conjugacy in $G\simeq GL(2,\Zb)  \ltimes \Rb^2$ (Theorem~\ref{FormasNormais}) and then describe the dynamics of  those normal forms  which are relevant to our study. 

In Section~\ref{dynamics},  the dynamics of the family of maps $\varphi_r$ is described.  This is done by obtaining first the dynamics of the restricted maps from the results of the previous section.

\section{Reduction and  restriction}\label{reduction}

Each map $\varphi_r$ of the family under study, that is  
$$\varphi_r(x_1,x_2,x_3,x_4)=  \left(x_3,x_4,\frac{x_2^r+x_3^r}{x_1}, \frac{x_1^rx_4^r+ (x_2^r+x_3^r)^r}{x_1^rx_2}\right),$$
is a cluster map associated to a mutation-periodic quiver of period 2  and 4 nodes, in the sense introduced by Fordy and  Marsh in \cite{FoMa}. This quiver, denoted by $Q_r$, and its associated skew-symmetric matrix $B_r=[b_{ij}]$ are both displayed in Figure~\ref{quiver}: the labels on the arrows of the oriented graph $Q_r$ denote the number of arrows between the corresponding nodes, and each entry $b_{ij}$ of $B_r$  is the number of arrows from node $i$ to node $j$ minus the number of arrows from $j$ to $i$ where the nodes are numbered as $(A,B,C,D)=(1,2,4,3)$.
For the notion of mutation-periodicity and the detailed construction of the respective cluster map we refer to   \cite{FoMa}, \cite{FoHo2}, \cite{InEs2} and \cite{Marsh}. 
        
\begin{figure}[h]
\begin{center}
{\includegraphics[scale=1]{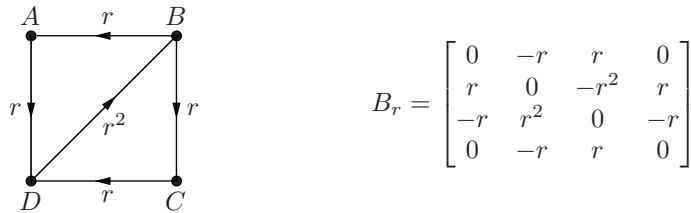}}
\caption{The quiver $Q_r$ and the associated skew-symmetric matrix $B_r$.} 
\label{quiver}
\end{center}
\end{figure}

In this section we first  show that each map $\varphi_r$ is semiconjugate to a  (parameter independent) map $\psi$ which is globally 4-periodic. Then we explain how this reduction  allows us to study the dynamics of $\varphi_r$ by studying the restrictions of  $\varphi_r$ and of its fourth iterate, $\varphi_r^{(4)}$, to certain 2-dimensional subvarieties  of $\Rb_+^4$ (cf. Proposition~\ref{orbits}). Explicit expressions of these restrictions are given in Proposition~\ref{Restrictions}, and it is shown that  all of them belong to a  group $\Gamma$  of symplectic birational maps of the plane.

\subsection{Reduction to a globally periodic map}

The reduction of $\varphi_r$ to a globally periodic map is achieved in two steps. First we use the reduction procedure  in \cite{InEs2} to semiconjugate $\varphi_r: \Rb_+^4\rightarrow \Rb_+^4$  to a symplectic map $\widehat{\varphi}_r:\Rb^2_+\rightarrow \Rb^2_+$. Second, we show that each symplectic reduced map $\widehat{\varphi}_r$ is (topologically) conjugate to a parameter-independent map $\psi$ which is  globally 4-periodic. These steps can be schematically summarized as follows:
$$\xymatrix{ \Rb_+ ^4  \ar @{->} [r]^{\Pi_r}\ar @{->} [d]_{\varphi_r} & \Rb_+ ^2 \ar @{->}[d]^{\widehat{\varphi}_r}\ar @{->}[r]^{h_r} &\Rb_+ ^2\ar @{->}[d]^{\psi}\\
 \Rb_+ ^{4}  \ar @{->} [r]_{\Pi_r} &  \Rb_+ ^{2} \ar @{->}[r]_{h_r}&\Rb_+ ^2
 }$$
For the first step we use the fact that the mutation-periodicity of the quiver $Q_r$ represented by the skew-symmetric matrix $B_r=[b_{ij}]$ is equivalent (see \cite[Theorem~3.1]{InEs2}) to the statement that  $\varphi_r$ preserves the \emph{standard presymplectic form}
\begin{equation}
\label{omegar}
\omega_r = \sum_{1\leq i<j\leq 4} \frac{b_{ij}} {x_ix_j} dx_i\wedge dx_j.  
\end{equation}
As the rank of $B_r$ is equal to 2, the reduced symplectic map $\widehat{\varphi}_r$ is obtained by choosing coordinates such that  the presymplectic form $\omega_r$ reduces to the {\it canonical symplectic} form
$$
\omega=\frac{1}{xy} dx\wedge dy.
$$
Following this procedure for the maps $\varphi_r$ in \eqref{Family1},  we obtain the following reduced maps in $\Rb_+^2$  
\begin{equation}\label{Reduced1}
\widehat{\varphi}_r(x,y)=\left(\frac{y \left(x^r+(1+y^r)^r\right)}{x^r}, \frac{1+y^r}{x}\right), 
\end{equation}
and the semiconjugacies $\Pi_r: \Rb_+^4\rightarrow \Rb_+^2$ 
\begin{equation}\label{semi1}
\Pi_r(x_1, x_2, x_3, x_4) = \left(\frac{x_1x_4}{x_2^r}, \frac{x_3}{x_2}\right).
\end{equation}
More details on these computations can be found in \cite[\textsection 4]{InEs2}.

\begin{proposition}\label{prop1}
Each map $\varphi_r: \Rb_+^4\rightarrow \Rb_+^4$ in (\ref{Family1}) is semiconjugate to the map $\psi: \Rb_+^2\rightarrow \Rb_+^2$ given by
\begin{equation}\label{Reduced2}
\psi(x,y)=\left(y,\frac{1}{x}\right)
\end{equation}
by the semiconjugacy  $\pi_r: \Rb_+^4\rightarrow \Rb_+^2$ defined as
 \begin{equation}\label{SC2}
\pi_r(x_1,x_2,x_3, x_4) =\left(\frac{x_3}{x_2},\frac{x_2^r+x_3^r}{x_1x_4}\right).
\end{equation}
That is,  $\pi_r\circ\varphi_r=\psi\circ\pi_r$.
\end{proposition}

\begin{proof} Each map $\varphi_r$ is semiconjugate to the symplectic reduced map $\widehat{\varphi}_r$ in \eqref{Reduced1} by the semiconjugacy $\Pi_r$ given by \eqref{semi1}. That is, the map $\Pi_r$ is surjective and satisfies $\Pi_r\circ\varphi_r=\widehat{\varphi}_r\circ\Pi_r$. Considering  the conjugacy (i.e. the homeomorphism) $h_r: \Rb_+^2\rightarrow \Rb_+^2$ defined by
 $$h_r (x,y)=\left(y, \frac{1+y^r}{x }\right),$$
it is easy to check  that  $h_r\circ\widehat{\varphi}_r=\psi\circ h_r$, with $\psi$ given by \eqref{Reduced2}. Taking $\pi_r=h_r\circ\Pi_r$ we  conclude that $\pi_r\circ\varphi_r=\psi\circ\pi_r$ as claimed.
\end{proof}

We note  that the map $\psi$ in \eqref{Reduced2} is  globally 4-periodic, that is $\psi^{(4)}=Id$ where $\psi^{(m)} = \psi\circ \cdots\circ \psi$ ($m$ compositions). 
\medskip

\subsection{Restriction to invariant sets}

We now describe the consequences of Proposition~\ref{prop1} to  the orbits of the map $\varphi_r$. 

Recall that, given a map $f: U\subset \Rb^m \rightarrow U$, a {\it periodic point $\mathbf{x}$ of $f$} is a point satisfying $f^{(m)}(\mathbf{x})=\mathbf{x}$ and $f^{(k)}(\mathbf{x})\neq\mathbf{x}$ for all $k<m$.  The number $m$ is called the {\it minimal period of ${\bf x}$}. The \emph{(forward) $f$-orbit} of $\mathbf{x}\in U$ is the set
$${\cal O}_f(\mathbf{x}) = \{ f^{(n)}(\mathbf{x}): n\in\Nb_0 \},$$
where $f^{(0)}$ denotes the identity map. 

The globally 4-periodic map $\psi$ in \eqref{Reduced2} has only one point with minimal period less than 4 which is the fixed point $(1,1)$. Moreover, the semiconjugacy between $\varphi_r$ and $\psi$ allows us to confine the orbits of $\varphi_r$ to certain  2-dimensional subsets of $ \Rb^4_+$.  This is a consequence of Theorem~2 in \cite{InHeEs} which is restated in the next proposition for the particular case of the maps $\varphi_r$.
 \medskip

\begin{proposition}
\label{orbits} 
Let  $\varphi_r:\Rb_+^4\rightarrow\Rb_+^4$ and $\psi :\Rb_+^2\rightarrow\Rb_+^2$ be the maps in Proposition~\ref{prop1}. For $(p,q)\in \Rb_+^2$ let 
\begin{equation}\label{Crpq}
C^r_{(p,q)}=\left\{(x_1,x_2,x_3,x_4)\in\Rb_+^4:\, x_3=p x_2,\,\, q x_1x_4= (1+ p^r) x_2^r\right\}.
\end{equation}
Then 
\begin{enumerate}
\item $C^r_{(p,q)}$ is invariant under $\varphi_r^{(4)}$; 
\item $C^r_{(1,1)}$ is invariant under $\varphi_r$.
\end{enumerate}
Moreover, if $(p,q)\neq (1,1)$ the sets $C^r_{\psi^{(i)}(p,q)}$, with $i=0,1, 2,3$, are pairwise disjoint and  the $\varphi_r$-orbit of  $\mathbf{x}\in C^r_{(p,q)}$ circulates cyclically  between  them  as follows
 $$C^r_{(p,q)}\longrightarrow C^r_{(q,1/p)}\longrightarrow C^r_{(1/p, 1/q)} \longrightarrow C^r_{(1/q, p)}\longrightarrow C^r_{(p,q)}
 $$
\end{proposition}

\begin{proof} 
Note that 
$$
C^r_{(p,q)} = \{ \mathbf{x} \in \Rb_+^4: \pi_r(\mathbf{x}) = (p,q)\}
$$
where $\pi_r$ is the map in \eqref{SC2}.  As $\pi_r\circ\varphi_r=\psi\circ\pi_r$ (cf. Proposition~\ref{prop1}) and $\psi$ is globally 4-periodic we have
$$\pi_r \circ \varphi_r^{(4)} (\mathbf{x})  = \psi^{(4)} \circ \pi_r (\mathbf{x}) = (p,q),$$
which means that $C^r_{(p,q)}$ is invariant under $\varphi_r^{(4)}$. 

The invariance of $C^r_{(1,1)}$ under $\varphi_r$ follows from the identity $\pi_r \circ \varphi_r(\mathbf{x})  = \psi \circ \pi_r (\mathbf{x})$ and from the fact that $(1,1)$ is a fixed point of $\psi$.

Finally the sets $C^r_{\psi^{(i)}(p,q)}$, with $i\in\{ 0,1, 2,3\}$, are clearly disjoint and for $\mathbf{x}\in C^r_{(p,q)}$ one has
$$\pi_r\circ \varphi_r (\mathbf{x}) = \psi (p,q),$$
which means that $\varphi_r (\mathbf{x})$ belongs to $C^r_{\psi(p,q)}$.
\end{proof}

Proposition~\ref{orbits}  implies that  each orbit of $\varphi_r$  is either entirely contained in the 2-dimensional algebraic variety $C^r_{(1,1)}$  or circulates between four pairwise disjoint algebraic varieties of dimension 2. So the dynamics of $\varphi_r$ can be studied through the restriction of $\varphi_r$ to $C^r_{(1,1)}$ and through the restrictions of $\varphi_r^{(4)}$ to  sets $C^r_{(p,q)}$ with $(p,q)\neq (1,1)$. In the next proposition we obtain explicit expressions of these restricted maps.

\begin{proposition}\label{Restrictions}
Let  $\varphi_r$ denote the map \eqref{Family1} and $C^r_{(p,q)}$  denote the algebraic variety of dimension 2 defined by  \eqref{Crpq}. Then,
\begin{enumerate}
\item $C^r_{(1,1)}$ is invariant under $\varphi_r$ and the restriction ${\bar\varphi}_r =\varphi_r\restr{C^r_{(1,1)}}$ is given in the coordinates $(x_1,x_2)$ by
\begin{equation}\label {Restr1}
\bar{\varphi}_r(x_1,x_2)=\left(x_2,2\frac{x_2^r}{x_1}\right).
\end{equation}
\item  $C^r_{(p,q)}$  is invariant under $\varphi_r^{(4)}$ and the restriction ${\widetilde\varphi}_r=\varphi^{(4)}_r\restr{C^r_{(p,q)}}$  is given in  the coordinates $(x_1,x_4)$ by
\begin{equation}\label{Restr4}
\widetilde{\varphi}_{r}(x_1,x_4)=\left(\lambda \frac{x_4^{r^2-2}}{x_1},\lambda^{r^2-1}\frac{x_4^{(r^2-3)(r^2-1)}}{x_1^{r^2-2}}\right),
\end{equation}
with 
\begin{equation}\label {lambda}
\lambda=\frac{(1+p^r)^2 (1+q^r)^r}{q^2 p^r}.
\end{equation}

\end{enumerate}
\end{proposition}
\begin{proof}
The restricted map ${\bar\varphi}_r$ is  obtained by a straightforward computation.  To obtain $\widetilde{\varphi}_{r}$, note that the computation of 
$\varphi_r^{(4)}(x_1,x_2,x_3,x_4)= (u_1,u_2,u_3,u_4)$ gives  
\begin{align*}
u_1&=l(\mathbf{x}) \frac{x_4^{r^2-2}}{x_1}, &\quad u_2&=l^r(\mathbf{x})\frac{x_2x_4^{r^3-3r}}{x_1^r},\\
u_3&=l^r(\mathbf{x})\frac{x_3x_4^{r^3-3r}}{x_1^r},&\quad u_4&=l^{r^2-1}(\mathbf{x})\frac{x_4^{(r^2-3)(r^2-1)}}{x_1^{r^2-2}},\\
\end{align*}
where 
\begin{equation}\label{Flambda}
l(\mathbf{x})=\frac{(x_1^rx_4^r+(x_2^r+x_3^r)^r)^r}{x_1^{r^2-2}x_2^rx_3^rx_4^{r^2-2}}.
\end{equation}
It is easy to see that the function $l$  is constant on   each $C^r_{(p,q)}$ and given by
$$\lambda = l(\mathbf{x})\restr{C^r_{(p,q)}}=\frac{(1+p^r)^2 (1+q^r)^r}{q^2 p^r}.$$ 
This leads directly to the expression of $\widetilde{\varphi}_{r}$ in the coordinates $(x_1,x_4)$. 
\end{proof}

\begin{remark} In the above proposition, the different choice of coordinates for $C^r_{(1,1)}$ and for $C^r_{(p,q)}$  has no particular meaning other than leading in each case to simpler expressions of the restricted maps. 
\end{remark}

It is easy to verify that the restricted maps \eqref{Restr1} and \eqref{Restr4} belong to  the group of maps from $\Rb^2_+$ to itself of the form
$$
f(x,y)=\left(\alpha x^ay^b,\beta x^cy^d\right),
$$
with $\alpha$ and $\beta$ real positive constants and  $a,b,c,d$ integers satisfying $ad-bc=1$. This group will be denoted by  $\Gamma$. The maps of $\Gamma$ are birational (rational with rational inverse) and preserve the symplectic form 
 \begin{equation}
 \label{w}
 \omega=\frac{1}{xy} dx\wedge dy,
 \end{equation}
in the sense that the pullback of $\omega$ by $f$ preserves $\omega$, that is $f^*\omega=\omega$.

\section{The group $\Gamma$}\label{gamma}

Using algebraic geometry techniques, it was proved in \cite{Blanc} that  the group of birational transformations of $\Cb^2$ preserving the symplectic form \eqref{w}   is generated by $SL(2, \Zb)$, the complex torus $(\Cb^*)^2$ and the globally 5-periodic (Lyness) map $(x,y)\mapsto (y, \frac{1+y}{x})$.  

In this section we study the group of birational maps $\Gamma$ and show that it is isomorphic to the semidirect product $SL(2,\Zb)\ltimes \Rb^2$.  We obtain  simplified forms ({\it normal forms}) for the maps of $\Gamma$ under conjugation and  describe the dynamics of those normal forms  which are relevant to our study.
 
\medskip
Let $\Gamma$ be the group of maps $ f:\Rb_+^2 \rightarrow \Rb^2_+$ defined by
\begin{equation} 
\label{f}
f(x,y)=\left(\alpha x^ay^b,\beta x^cy^d\right), \quad  \alpha, \beta \in \Rb_+, \;\; ad-bc=1, \,\,a,b,c,d\in\Zb.
\end{equation}
The map $i:\Rb_+^2\longrightarrow \Rb^2$ given by 
\begin{equation}\label{iso}
i(x,y)=(\log x,\log y)
\end{equation}
conjugates $f\in \Gamma$ to the affine map  in $\Rb^2$
$$g(u,v) =  ( a u + bv + \log \alpha,  c u + dv + \log \beta).$$
Note that $g$  is the composition of the translation   by the vector $\mathbf{v}= (\log \alpha,\log\beta)$ and the area preserving linear map represented by the $SL(2, \Zb)$ matrix $M=\begin{bmatrix}
a&b\\
c&d
\end{bmatrix}$.   Identifying $g$ with $(M,\mathbf{v})$, the map $i$  induces  an isomorphism between $\Gamma$ and the semidirect product
$$SL(2,\Zb)\ltimes \Rb^2= \{ (M,\mathbf{v}): M\in SL(2,\Zb), \mathbf{v} \in \Rb^2  \}$$
with group multiplication defined by $(M,\mathbf{v})\cdot (N,\mathbf{w})= (MN, \mathbf{v}+M\mathbf{w})$.

\subsection{Normal forms} 

In the next theorem we  obtain simplified forms for elements of $\Gamma$  up to conjugation.  We show that apart from the following maps 
$$f_{(\alpha,\beta)}^{\pm}(x,y)=(\alpha x^{\pm 1}, \beta y^{\pm 1})$$
there are only two types of normal forms characterising the elements of $\Gamma$. 

Note that the maps $f_{(\alpha,\beta)}^{\pm}$ (corresponding in \eqref{f} to  $b=c=0$) do not appear in Proposition~\ref{Restrictions} as restricted maps and furthermore their dynamical behaviour is trivial. In fact, the map $f_{(\alpha,\beta)}^+$ is conjugate to a translation and $f_{(\alpha,\beta)}^-$ is globally 2-periodic. 

\begin{theorem}\label{FormasNormais}
Let
$$f(x,y)=\left(\alpha x^ay^b,\beta x^cy^d\right), \quad \alpha, \beta>0, \;\; ad-bc=1,$$
be an element of $\Gamma$ with $b^2+c^2\neq 0$. Then $f$ is conjugate  to one of the following maps
\begin{enumerate}
\item $f_k(x,y)=(y,\frac{y^k}{x})$ with  $k=a+d$ if $a+d\neq 2$;
\item $f_{2,\xi}(x,y)=(y,\xi\frac{y^2}{x})$ if $a+d=2$, where 
\begin{equation}\label{xi}
\xi=\begin{cases}
{\displaystyle \frac{\alpha^c}{\beta^{a-1}}},& \text{if $c\neq 0$}\\
 \beta^b, & \text{if $c=0$.}
\end{cases}
\end{equation}

\end{enumerate}
\end{theorem}

\begin{proof}  If $c\neq 0$,  considering the homeomorphism  $\pi$ of $\Rb_+^2$ given by 
 $$\pi(x,y)=(y^ax^{-c},\beta^a\alpha^{-c}y),$$
it is easy to check that $\pi\circ f=g\circ\pi$, where $g$ is the map 
$$g(x,y)=(y,K \frac{y^{a+d}}{x})\quad {\rm with} \quad K=\beta(\beta^a\alpha^{-c})^{1-(a+d)}.$$
If $a+d=2$ the map $g$ is the map $f_{2,\xi}$ with $\xi= \frac{\alpha^c}{\beta^{a-1}}$ . If $a+d\neq 2$, taking the following map $\Pi$ 
$$\Pi(x,y)=K^{\frac{1}{a+d-2}}(x,y),$$
we have $\Pi\circ g=f_{a+d}\circ\Pi$, that is  $\Pi\circ\pi\circ f=f_{a+d}\circ\Pi\circ\pi$.

If $c=0$, the hypothesis $b^2+c^2\neq 0$  implies that $b\neq 0$. Considering  the  involution $\sigma(x,y)=(y,x)$, which interchanges $c$ and $b$,  the problem reduces  to the previous cases. In fact,  $\sigma\circ f\circ \sigma=\left(\beta x^d,\alpha x^by^a\right)$ is conjugate to $f_{a+d}$ if $a+d\neq 2$ and to $f_{2,\xi}$ with $\xi= \frac{\beta^b}{\alpha^{d-1}}=\beta^b$ if $a+d=2$. 
\end{proof}

\begin{remark} 
\label{remclass} It is worth noting that the conjugacies in the proof of the above theorem belong to a group $G$ which is isomorphic to $GL(2,\Zb)\ltimes \Rb^2$. The result in the theorem may be rephrased as follows. Up to conjugation in $G$, the elements $(M,\mathbf{v})\in SL(2,\Zb)\ltimes \Rb^2$, with  $M\neq \pm I$, are parametrized by  the trace of $M$ if $\operatorname{tr}\, M\neq 2$, and by a real parameter $\xi$ which depends on $M$ and $\mathbf{v}$ through the expression \eqref{xi} if $\operatorname{tr}\, M=2$.

As we will see the explicit conjugacies given in  the proof of Theorem~\ref{FormasNormais}  (and in Corollary~\ref{cor1} below) play a key role in the study of the dynamics performed in the following sections. We believe that if one was only interested  in the normal forms   {\it per se}, these might be obtained by quoting  results from group theory scattered in the literature. 
\end{remark}

\bigbreak
As  a consequence of  Theorem~\ref{FormasNormais} the  restricted maps (\ref{Restr1}) and (\ref{Restr4})  are conjugate to the normal forms given in the following corollary.

\begin{corollary}\label{cor1}
Let $r$ and $\lambda$ be a positive integer and a positive real number, respectively. Consider the maps
$$\bar{\varphi}_r(x,y)=(y,2\frac{y^r}{x}), \qquad \widetilde{\varphi}_{r}(x,y)=\left(\lambda \frac{y^{r^2-2}}{x},\lambda^{r^2-1}\frac{y^{(r^2-3)(r^2-1)}}{x^{r^2-2}}\right).$$
\begin{enumerate}
\item If $r=2$, then
\begin{itemize}
\item[i)]  $\bar{\varphi}_2$ is already in normal form: $\bar{\varphi}_2=f_{2,2}$;
\item[ii)] $\widetilde\pi_2\circ \widetilde{\varphi}_2= f_{2, \lambda^4}\circ\widetilde\pi_2$ with
\begin{equation}\label{eq:tildefiproj1}
 \widetilde \pi_2(x,y)= \left(\frac{x^2}{y},  \frac{y}{\lambda}\right),\qquad
  f_{2, \lambda^4}(x,y)= \left(y, \lambda^4 \frac{y^2}{x}\right).
\end{equation}
\end{itemize}
\item If $r\neq 2$ then,
\begin{itemize}
\item[i)] $\bar{\pi}_r\circ \bar{\varphi}_r= f_r\circ \bar{\pi}_r$ with
\begin{equation}\label{eq:barfiproj}
\bar{\pi}_r(x,y)=  2^{\frac{1}{r-2}} ( x, y), \qquad
 f_r(x,y)= \left(y,   \frac{y^r}{x}\right);
\end{equation}

\item[ii)] $\widetilde{\pi}_r\circ \widetilde{\varphi}_r= f_{(r^2-2)^2-2}\circ \widetilde{\pi}_r$ with
\begin{align}
\label{eq:tildefiproj}
\widetilde{\pi}_r(x,y)  = \lambda^{\frac{1}{r^2-4}} \left(\lambda \frac{x^{r^2-2}}{y}, y\right) &,\nonumber \\ 
f_{(r^2-2)^2-2}(x,y) = \left(y, \frac{y^{(r^2-2)^2-2}}{x}\right). &
\end{align}
\end{itemize}
\end{enumerate}
\end{corollary}

\begin{proof}
Note that both maps ${\bar\varphi}_r $ and $\widetilde\varphi_r $ verify the hypotheses of Theorem~\ref{FormasNormais} with $c\neq 0$, for any $r$. Furthermore,  $a+d=r$ for $\bar{\varphi}_r$ and  $a+d=(r^2-2)^2-2$ for $\widetilde{\varphi}_r$. Also, for both maps $a+d=2$ if and only if   $r=2$. The result then follows from the proof of Theorem~\ref{FormasNormais}.
\end{proof}

\subsection{Dynamics of $f_k$ and $f_{2,\xi}$}\label{dyngamma}

To understand the dynamics of the restricted maps \eqref{Restr1} and \eqref{Restr4} it is enough to analyse the dynamics of the maps $f_k$ and $f_{2,\xi}$ given in Theorem~\ref{FormasNormais}. The dynamics of these maps can be better described by using first integrals and the graphical representation of their level sets.

A \emph{first integral} of a  given map $f: U\subset \Rb^m \rightarrow U$ is a non constant function $I: U \rightarrow \Rb$ which is constant on $f$-orbits, that is
$$I \circ f(\mathbf{x}) = I(\mathbf{x}), \quad \text{for all $\mathbf{x}\in U$}.$$
Any $f$-orbit is therefore confined to a level set of a first integral  of $f$.

\subsubsection*{Dynamics  of $f_k(x,y)=(y, \frac{y^k}{x}),\,\, k\neq 2$}

We now study the dynamics of the  maps $f_k(x,y)=(y, \frac{y^k}{x})$, with $k\neq 2$. Although this  study can be performed for any integer $k$, we restrict (in Lemma~\ref{lemafk} below) to the integers $k\geq -1$  since smaller values do not appear in the restricted maps of Proposition~\ref{Restrictions}.

The dynamics in $\Rb^2_+$ of the map $f_k(x,y)=(y, \frac{y^k}{x})$ is easily obtained, via the isomorphism \eqref{iso}, from the properties of the linear map 
\begin{equation}
\label{gk}
g_k(u,v) = (v, -u+kv),
\end{equation}
which is represented by a $SL(2,\Zb)$ matrix. Such linear maps  leave invariant a quadratic form $Q_k$ (see for instance \cite{LagRa}).   Composing this quadratic form with the isomorphism $i(x,y)=(\log x, \log y)$ gives the following first integral of the map $f_k$: 
\begin{equation}\label{int}
I_k(x,y) = \log ^2(x)-k\log(x)\log (y)+\log^2(y). 
\end{equation}

Consequently,  each orbit of $f_k$  lies on a level set of  $I_k$. These level sets are displayed in Figure~\ref{fig2} for $k<2$ and in Figure~\ref{fig3} for $k>2$.   Further information on the orbits of $f_k$ can be obtained directly from the well known dynamical properties of the  linear  map $g_k$ in \eqref{gk}.  The next lemma summarizes the  dynamics of the map $f_k$.

\begin{figure}[htb]
\begin{center}
{\includegraphics[scale=0.8]{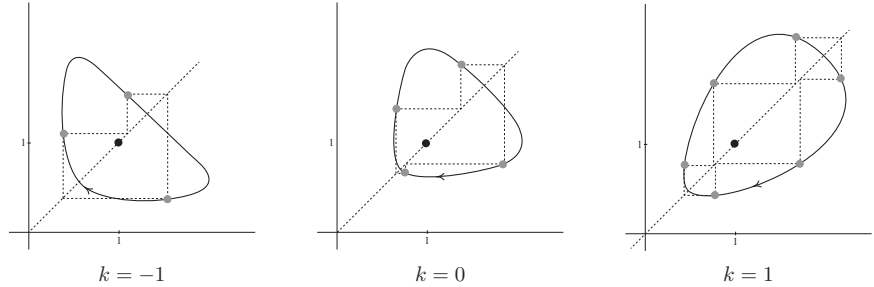}}
\caption{ Level sets of $I_k$ for  $k\in \{-1,  0,1\}$ and graphical construction of the $f_k$-orbit in the level set. The arrows indicate how the forward orbit moves along the level set.} \label{fig2}
\end{center}
\end{figure}

\begin{figure}[htb]
\begin{center}
{\includegraphics[scale=0.9]{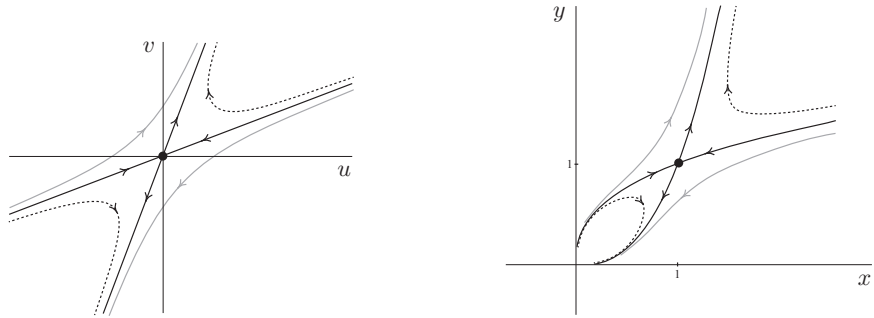}}
\caption { On the left the level sets of the quadratic form $Q_k$ for  $k>2$ and on the right the corresponding level sets of $I_k$. The arrows indicate how the forward orbit moves along the level set.}\label{fig3}
\end{center}
\end{figure}

\begin{lemma}
\label{lemafk}
Let $f_k: \Rb_+^2 \rightarrow \Rb_+^2$  be given by $f_k(x,y)=(y, \frac{y^k}{x})$ with   $k\in \Zb$. Then,
\begin{enumerate}
\item $f_{-1}$ is globally $3$-periodic,  $(1,1)$ is  its unique fixed point and all the other points have minimal period 3.
\item $f_{0}$ is globally $4$-periodic, $(1,1)$ is its unique fixed point and all the other points have minimal period 4.
\item $f_{1}$ is globally 6-periodic, $(1,1)$ is its unique fixed point and all the other points have minimal period 6.
\item If $k>2$ then $(1,1)$ is the unique fixed point of $f_k$ and there are no more  periodic points. Moreover, each $f_k$-orbit remains in the same connected component of a level curve of the first integral \eqref{int}, and:
\begin{enumerate}
\item the $f_k$-orbit of $P\in\Rb^2_+$ converges to $(1,1)$ if $P$ belongs to the subset of $\Rb_+^2$ defined  by
\begin{equation}\label{setF}
{\cal F}_k=\{I_k(x,y)=0\} \cap \{ (x<1, \, x<y)\,  \text{or}\, (x>1, \, x>y) \}
 \end{equation}
\item the $f_k$-orbit of $P\in\Rb^2_+$ converges to $(0,0)$ if $P$ belongs to the subset of $\Rb_+^2$ defined  by
\begin{align}\nonumber \label{setZ}
{\cal Z}_k&=\left\{I_k(x,y)<0, \, x<1\right\} \cup \left\{I_k(x,y)=0,\, x<1, \,x>y \right\} \\
&\qquad \cup \left\{I_k(x,y)>0,\, x>y \right\}.
 \end{align}
\item in any other case both components of $f_k^{(n)}(P)$ go to $+\infty$.
\end{enumerate}

\end{enumerate}
\end{lemma}

\subsubsection*{Dynamics  of $f_{2,\xi}(x,y)=(y, \xi \frac{y^2}{x})$}

The dynamics of the map $f_{2,\xi}(x,y)=(y,\xi\frac{y^2}{x})$ with $\xi\in \Rb_+$ is obtained from the general expression of  the iterates $f^{(n)}_{2,\xi}(x,y)$. This expression can be computed by applying Lemma 1 in \cite{InHeEs}.   An alternative way is to consider the conjugate affine map 
$$g_{2,\xi}(u,v) = (v, -u+2v+\log \xi),$$
which can be identified with the $SL(3,\Rb)$ matrix 
$$X=\left[\begin{array}{cc|c}
0&1&0\\
-1&2&\log \xi\\
\hline
0&0&1
\end{array}\right].$$
By computing the $n$th power of $X$ we arrive at the expression of $g_{2,\xi}^{(n)}$ from which we derive:
$$
f_{2,\xi}^{(n)}(x,y)=\xi^{\frac{n(n-1)}{2}}\frac{y^n}{x^n}\left(x,\xi^ny\right), \quad n\geq 0.
$$
The consequences for the dynamics of $f_{2,\xi}$ follow immediately from this expression and are summarised in the next lemma. 

\begin{lemma}\label{lemaf2xi}
Let  $f_{2,\xi}:\Rb_+^2\longrightarrow \Rb_+^2$ be given by $f_{2,\xi}(x,y) = \left (y, \xi\frac{y^2}{x}\right )$ with $\xi\in  \Rb_+$. 
\begin{enumerate}
\item  If $\xi<1$,  then the map $f_{2,\xi}$ has  no periodic points and the $f_{2,\xi}$-orbit of any point converges to $(0,0)$.
\item  If $\xi=1$,   then  all the points of the form $(x,x)$ are fixed points of  $f_{2,\xi}$. Moreover  $f_{2,\xi}^{(n)}(x,y)$ converges to $(0,0)$ if $y<x$ and both components of   $f_{2,\xi}^{(n)}(x,y)$ go to $+\infty$ if $y>x$.
\item  If $\xi>1$,   then  $f_{2,\xi}$ has  no periodic points and both components of $f_{2,\xi}^{(n)}(P)$ go to $+\infty$ for any $P\in \Rb_+^2$.
\end{enumerate}
\end{lemma}

\begin{remark}
 The maps $f_k$ and $f_{2,\xi}$ in Theorem~\ref{FormasNormais} are reversible symplectic birational maps. In fact, the map  $R(x,y)=(y,x)$ is an involutory reversing symmetry of both $f_k$ and $f_{2,\xi}$, that is
$$R\circ h \circ R^{-1} = h^{-1}, \quad \text{for $h=f_k$ or $h=f_{2,\xi}$}.$$
with $R^2=Id$.  The consequences of this reversing symmetry to the  orbits of  $f_k$ and $f_{2,\xi}$ are clear from figures~\ref{fig2} and \ref{fig3}. 

Moreover as the restricted maps \eqref{Restr1} and \eqref{Restr4} are conjugate to one of these normal forms by a homeomorphism $\pi$, they also admit the involutory reversing symmetry $\pi^{-1}\circ R\circ \pi$.

 The existence of an involutory symmetry for a given map implies that the map can be written as a composition of two involutions \cite{LaRo}. This property plays a key role in the studies \cite{Duist} and \cite{IaRo} of the QRT maps \cite{QRT}, which are birational maps of the plane admitting a biquadratic  first integral.

 Reversing symmetries of maps have been widely studied and have strong consequences to the dynamics of the map. We refer to  \cite{LaRo} and references therein for a survey on the subject. For works on reversing symmetries related to the context of symplectic (birational) maps,  we refer for instance to  \cite{BaRo}, \cite{KoMe}, \cite{RoBa} and  \cite{Ro}.

\end{remark}

\section{ Dynamics of the family $\varphi_r$}\label{dynamics}

The dynamics of the family $\varphi_r$ given by \eqref{Family1}   is based on the dynamics of the restricted maps $\bar{\varphi}_r$ and $\widetilde{\varphi}_r$ in Proposition~\ref{Restrictions}. In turn, the dynamics of these restricted maps is easily obtained from the results in the previous section.

\medskip

\subsection{ Dynamics of the restricted maps}

In this subsection we combine the results in Corollary~\ref{cor1} with those in lemmas~\ref{lemafk} and \ref{lemaf2xi}, to describe the dynamics of the restricted maps $\bar{\varphi}_r$  and $\widetilde{\varphi}_r$  (below, in propositions~\ref{propbar} and \ref{proptilde}, respectively). 

For future reference, we  recall a basic  property of conjugate maps. If two maps $f$ and $g$ are conjugate by a homeomorphism $\pi$, that is $\pi \circ g = f\circ \pi$, then the dynamics of $g$ can be obtained from the dynamics of $f$ through the identity
\begin{equation}
\label{dyn}
g^{(n)} = \pi^{-1}\circ f^{(n)} \circ \pi, \quad n=1, 2, \ldots
\end{equation}

\begin{proposition}
\label{propbar} 
Let $r$ be a positive integer,   $\bar{\varphi}_r: \Rb_+^2 \rightarrow \Rb_+^2$ the map
$$\bar{\varphi}_r(x,y)=\left(y,2\frac{y^r}{x}\right),$$
 $\bar{\pi}_r$ the homeomorphism given by $\bar{\pi}_r(x,y) = 2^{\frac{1}{r-2}} \left(x,  y\right)$ with $r\neq 2$, and ${\cal F}_r, {\cal Z}_r\subset\Rb_+^2$ the sets defined by \eqref{setF} and \eqref{setZ}.
Then,
\begin{itemize}
\item[i)]  $\bar{\varphi}_1$ is globally 6-periodic, the point  $(2,2)$ is a fixed point of  $\bar{\varphi}_1$ and any other  point  $P\in\Rb_+^2$ is periodic with minimal period 6.
\item[ii)]  $\bar{\varphi}_2$ has no periodic points and for any $P\in\Rb_+^2$  each  component of $\bar{\varphi}_2^{(n)}(P)$ goes to $+\infty$.
\item[iii)] If $r>2$ the map $\bar{\varphi}_r$ has a unique fixed point $(2^{\frac{1}{2-r}}, 2^{\frac{1}{2-r}})$ and no other periodic points. 

If $P\in \bar{\pi}_r^{-1}({\cal F}_r)$ then the $\bar{\varphi}_r$-orbit of $P$  converges to the fixed point. 

If $P\in \bar{\pi}_r^{-1}({\cal Z}_r)$ then the $\bar{\varphi}_r$-orbit of $P$  converges to $(0,0)$. 

For any  other point $P\in\Rb_+^2$  each  component of $\bar\varphi_r^{(n)}(P)$ goes to $+\infty$. \label{BHV1}
\end{itemize}
\end{proposition}
\begin{proof} Statement $ii)$ just  follows from Lemma~\ref{lemaf2xi} with $\xi=2$. 

For the other statements, note that  by Corollary~\ref{cor1} the map $\bar{\varphi}_r$ is conjugate to the map $f_r(x,y)= (y,y^r/x)$ by the conjugacy $\bar{\pi}_r$.  The  conclusions then follow from Lemma~\ref{lemafk} and  from the identity \eqref{dyn} by taking into account that 
$\bar{\pi}_r^{-1}(x,y) = 2^{\frac{1}{2-r}}\left( x, y\right)$.
\end{proof}

\begin{proposition}
\label{proptilde} 
Let $r$ be a positive integer  and  $\widetilde{\varphi}_r: \Rb_+^2 \rightarrow \Rb_+^2$   the map
$$\widetilde{\varphi}_r(x,y)=\left(\lambda \frac{y^{r^2-2}}{x},\lambda^{r^2-1}\frac{y^{ (r^2-3)(r^2-1)}}{x^{r^2-2}}\right),$$ 
where 
$$\lambda=\frac{(1+p^r)^2 (1+q^r)^r}{q^2 p^r}, \quad (p,q)\in \Rb_+^2\backslash\{(1,1)\}.$$
Consider  the homeomorphism  $\widetilde{\pi}_r(x,y)  =\lambda^{\frac{1}{r^2-4}} \left(\lambda \frac{x^{r^2-2}}{y}, y\right)$  with $r\neq 2$, and ${\cal F}_r, {\cal Z}_r\subset\Rb_+^2$  the sets  defined by \eqref{setF} and \eqref{setZ}.
Then,
\begin{itemize}
\item[i)]  $\widetilde{\varphi}_1$ 
 is globally 3-periodic,  the point  $(\lambda^{1/3},\lambda^{1/3})$ is its unique fixed point   and any other point has minimal period 3.
 \item[ii)]  $\widetilde{\varphi}_2$ has no periodic points and both components of $\widetilde{\varphi}_2^{(n)}(P)$ go to $+\infty$ for any $P\in\Rb_+^2$.
\item[iii)] For $r>2$, the map $\widetilde{\varphi}_r$ has a unique fixed point $\left(\lambda^{\frac{1}{4-r^2}}, \lambda^{\frac{1}{4-r^2}}\right)$ and no other periodic points. 

If $P\in \widetilde{\pi}_r^{-1}({\cal F}_{(r^2-2)^2-2})$ the $\widetilde{\varphi}_r$-orbit of $P$ converges to the fixed point. 

If $P\in \widetilde{\pi}_r^{-1}({\cal Z}_{(r^2-2)^2-2})$ then the $\widetilde{\varphi}_r$-orbit of $P$ converges to $(0,0)$. 

For any other point $P$,  both components of $\widetilde\varphi_r^{(n)}(P)$ go to $+\infty$.
\end{itemize} 
\end{proposition}

\begin{proof}  By Corollary~\ref{cor1}-{\it 1.ii)}, the map $\widetilde{\varphi}_2$ is conjugate to the map $f_{2,\lambda^4}$ by the conjugacy $\widetilde \pi_2$  in \eqref{eq:tildefiproj1}. Note that in this case  $\lambda =\left(p+\frac{1}{p}\right )^2\left (q+\frac{1}{q}\right)^2>1$. Using \eqref{dyn} and noting that 
$\widetilde{\pi}_2^{-1}(x,y) = ( \sqrt{\lambda x y}, \lambda y)$, statement $ii)$  follows from Lemma~\ref{lemaf2xi}-{\it 3} with $\xi=\lambda^4$.  

If $r\neq 2$ then  by Corollary~\ref{cor1}-{\it 2.ii)} the map $\widetilde{\varphi}_r$ is conjugate to the map $f_{(r^2-2)^2-2}$ in  \eqref{eq:tildefiproj} by the conjugacy $\widetilde{\pi}_r$.

For statement $i)$, $f_{(r^2-2)^2-2}= f_{-1}$ and so the conclusion follows from Lemma~\ref{lemafk}-{\it 1}. 

Finally,  if $r>2$ then $(r^2-2)^2-2>2$ and so the last statement  follows from Lemma~\ref{lemafk}-{\it 4} using again the identity \eqref{dyn} and  the fact that $\widetilde{\pi}_r^{-1}(x,y) = \lambda^{\frac{1}{4-r^2}}\left( (xy)^{\frac{1}{r^2-2}}, y\right).$ 
\end{proof}

\subsection{Dynamics of $\varphi_r$}

We  now address the problem of describing the dynamics of the maps of the family \eqref{Family1}. Recall that the maps of this family are given by 
$$\varphi_r(x_1,x_2,x_3,x_4)=  \left(x_3,x_4,\frac{x_2^r+x_3^r}{x_1}, \frac{x_1^rx_4^r+ (x_2^r+x_3^r)^r}{x_1^rx_2}\right),$$
where the parameter $r$ is a positive integer. 

By Proposition~\ref{orbits}, each orbit of $\varphi_r$  is either entirely contained in the 2-dimensional algebraic variety $C^r_{(1,1)}$ or circulates between four pairwise disjoint 2-dimensional algebraic varieties: 
$$C^r_{(p,q)}, \quad C^r_{(q,p^{-1})}, \quad C^r_{(p^{-1}, q^{-1})}, \quad C^r_{(q^{-1}, p)}$$
 which are all invariant under $\varphi_r^{(4)}$. The dynamics of the restrictions of $\varphi_r$ to
$C^r_{(1,1)}$ and of $\varphi_r^{(4)}$ to $C^r_{P}$ with $P\neq (1,1)$ were obtained in Proposition~\ref{propbar} and Proposition~\ref{proptilde} respectively. 
 From the results in  those propositions  we  are able to describe the dynamics of the maps $\varphi_r$ and conclude that there are three different types of dynamical behaviour  depending on whether $r=1$, $r=2$ or  $r>2$. 
 
The dynamics of $\varphi_r$ in  the cases $r=1$ and $r>2$ is described  respectively in Theorem~\ref{tr=1} and Theorem~\ref{tr>2} below. 
The case $r=2$ will not be considered in what follows, since it corresponds to one of the maps whose dynamics was described in detail in  \cite[Theorem 3]{InHeEs}.  For the sake of completeness we sum up the results from that reference: $\varphi_2$ has no periodic points and for any $ \mathbf{x}\in\Rb_+^4$  each  component of $\varphi_2^{(n)}(\mathbf{x})$ goes to $+\infty$. The explicit expression of $\varphi_2^{(n)}$ was also obtained in the  referred work.

\begin{theorem}\label{tr=1}
The map $\varphi_1:\Rb^4_+\rightarrow \Rb^4_+$  given by
$$\varphi_1(x_1,x_2,x_3,x_4)=  \left(x_3,x_4,\frac{x_2+x_3}{x_1}, \frac{x_1x_4+ x_2+x_3}{x_1x_2}\right)$$
is globally 12-periodic. Moreover, 
\begin{enumerate}
\item $\varphi_1$ has exactly one fixed point, the point $F=(2,2,2,2)$.
\item Every point in the  (punctured) algebraic variety $V\backslash\{F\}$  where
$$V=\{(x_1,x_2,x_3,x_4)\in\Rb_+^4: x_4=x_1,\,\, x_1x_2x_3=x_1^2+x_2+x_3\},$$
is periodic with minimal period 4. 
\item Every point in the  (punctured) algebraic variety $C^1_{(1,1)}\backslash\{F\}$  where
$$C^1_{(1,1)}=\{(x_1,x_2,x_3,x_4)\in\Rb_+^4: \,\,x_3=x_2,\, \,  x_1x_4=2x_2\},$$
is periodic with minimal period 6.
\end{enumerate}
Any other point has minimal period 12.
\end{theorem}

\begin{proof}
By Proposition~\ref{orbits}, the $\varphi_1$-orbit of any point is either entirely contained in the algebraic variety $C^1_{(1,1)}$ or moves between four pairwise disjoint algebraic varieties
$$C^1_{(p,q)}, \quad C^1_{(q,p^{-1})}, \quad C^1_{(p^{-1},q^{-1})}, \quad C^1_{(q^{-1},p)}$$
with  $(p,q)\in  \Rb^2_+\backslash \{(1,1)\}$.  

The restriction $\bar\varphi_1$ of $\varphi_1$ to $C^1_{(1,1)}$ is globally 6-periodic by Proposition~\ref{propbar}, and the restriction $\widetilde\varphi_1$ of $\varphi_1^{(4)}$ to any $C^1_{(p,q)}$ is globally 3-periodic by Proposition~\ref{proptilde}. Hence $\varphi_1$ is globally 12-periodic. 
Moreover, all the points in $C^1_{(1,1)}$ have minimal period 6 except the point $F=(2,2,2,2)$ which is fixed. 

Any point not belonging to $C^1_{(1,1)}$ is either a fixed point of $\varphi_1^{(4)}$ or a periodic point of $\varphi_1^{(4)}$ with minimal period 3. 

To compute the fixed points of $\varphi_1^{(4)}$, which correspond to periodic points of $\varphi_1$ with minimal period 4, we refer again to Proposition~\ref{proptilde}-{\it i)}. Each of these points $\mathbf{x}$ belongs to a set $C^1_{(p,q)}$ with $(p,q)\neq (1,1)$  and its coordinates $(x_1,x_4)$ satisfy $x_1=x_4=\lambda^{1/3}$, where $\lambda$ is given by \eqref{lambda} (with $r=1$). On the other hand the constant $\lambda$ is the value of the restriction to $C^1_{(p,q)}$ of the function $l(\mathbf{x})$ given in \eqref{Flambda}. To obtain the set $V$ it is enough to eliminate $\lambda$ from these relations, that is from 
$$x_1=x_4=\lambda^{1/3}, \quad \lambda=\frac{x_1 x_4+x_2+x_3}{x_1^{-1}x_2x_3x_4^{-1}}.
$$

Finally the remaining points are periodic points of $\varphi_1^{(4)}$ with minimal period 3, and therefore they are periodic points of $\varphi_1$ with minimal period 12. 
\end{proof}

\begin{theorem}\label{tr>2}
 For each integer $r>2$,  let  $\varphi_r:\Rb^4_+\rightarrow \Rb^4_+$ be the map
$$\varphi_r(x_1,x_2,x_3,x_4)=  \left(x_3,x_4,\frac{x_2^r+x_3^r}{x_1}, \frac{x_1^rx_4^r+ (x_2^r+x_3^r)^r}{x_1^rx_2}\right)$$
and  $\pi_r:\Rb^4_+\rightarrow \Rb^2_+$ the map
$\pi_r(\mathbf{x}) =\left(\frac{x_3}{x_2},\frac{x_2^r+x_3^r}{x_1x_4}\right)$. 
Then,
\begin{enumerate}
\item $\varphi_r$ has a unique fixed point, the point $F=(2^{\frac{1}{2-r}},2^{\frac{1}{2-r}},2^{\frac{1}{2-r}},2^{\frac{1}{2-r}})$. 
\item Any point  in the (punctured) algebraic variety $V\backslash\{F\}$ where
$$V=\left\{(x_1,x_2,x_3,x_4)\in\Rb_+^4: \, x_4=x_1, \,\, x_1^rx_2x_3=x_1^{2r}+\left(x_2^r+x_3^r\right)^r\right\},$$
is a periodic point of $\varphi_r$ with minimal period 4.

\item Any  point $\mathbf{x}\notin V$ is non-periodic and

\begin{enumerate}
\item if $\mathbf{x}$ belongs to the $\varphi_r$-invariant set
$$C^r_{(1,1)}=\{(x_1,x_2,x_3,x_4)\in\Rb_+^4: x_3=x_2,\,\, x_1x_4=2x_2^r\},$$
 the $\varphi_r$-orbit of  $\mathbf{x}$ is entirely contained in the curve:
 $$L_\mathbf{x}=\{{\bf z}\in C^r_{(1,1)}: J({\bf z})=J(\mathbf{x})\},$$
where $J({\bf z})=\log^2(z_1)-r\log(z_1)\log(z_2)+\log^2(z_2)-\log(2)\log(z_1z_2)$. 

Moreover, the orbit of $\mathbf{x}$ either converges to the fixed point $F$ or converges to $(0,0,0,0)$ or else every component of $\varphi_r^{(n)}(\mathbf{x})$ goes to $+\infty$.

\item if $\mathbf{x}$ does not belong to $C^r_{(1,1)}$ then its $\varphi_r$-orbit circulates  between  the four pairwise disjoint curves ${\cal L}_{\bf{x}}$, ${\cal L}_{\varphi_r(\bf{x})}$, ${\cal L}_{\varphi_r^{(2)}(\bf{x})}$ and ${\cal L}_{\varphi_r^{(3)}(\bf{x})}$  where
$${\cal L}_{\bf y}=\{{\bf z}\in \Rb_+^4: \pi_r({\bf z})=\pi_r({\bf y}), \,\, J({\bf z})=J({\bf y})\},$$
with 
$$J({\bf z})=\log^2(z_1z_4)-r^2\log (z_1)\log (z_4)-\log(l({\bf z}))\log (z_1z_4),$$
 and $l$ the function defined in  \eqref{Flambda}.

Moreover, the orbit of $\mathbf{x}$ either converges to the orbit of a 4-periodic point, or converges to $(0,0,0,0)$ or else every component of $\varphi_r^{(n)}(\mathbf{x})$ goes to $+\infty$. 
\end{enumerate}
\end{enumerate}
\end{theorem}

\begin{proof}  The proof follows the same lines as the proof of Theorem~\ref{tr=1} substituting for each  $(p,q)\in  \Rb^2_+$ the set $C^1_{(p,q)}$ by  
$$C^r_{(p,q)}=\left\{(x_1,x_2,x_3,x_4)\in\Rb_+^4:\, x_3=p x_2,\,\, q x_1x_4= (1+ p^r) x_2^r\right\}.$$

As  $C^r_{(1,1)}$ is invariant under $\varphi_r$ and any $C^r_{(p,q)}$ is invariant under $\varphi_r^{(4)}$  (cf. Proposition~\ref{orbits}), the orbit of any point $\mathbf{x}\in\Rb^4_+$ can be studied by looking at the restriction $\bar\varphi_r$ of $\varphi_r$ to $C^r_{(1,1)}$ and at the restrictions $\widetilde\varphi_r$ of $\varphi_r^{(4)}$ to each $C^r_{(p,q)}$. The restriction $\bar\varphi_r$ is given in the coordinates $(x_1,x_2)$ by \eqref{Restr1} and  $\widetilde\varphi_r$ is given in the coordinates $(x_1,x_4)$ by \eqref{Restr4}.

By Proposition~\ref{propbar}-{\it iii)}, $(x_1,x_2)=(2^{\frac{1}{2-r}},2^{\frac{1}{2-r}})$ is the unique fixed point of the map $\bar\varphi_r$  corresponding to the fixed point $F=(2^{\frac{1}{2-r}},2^{\frac{1}{2-r}},2^{\frac{1}{2-r}},2^{\frac{1}{2-r}})$ of $\varphi_r$. 

By Proposition~\ref{proptilde}-{\it iii)},  for each $(p,q)\neq (1,1)$ the restriction  $\widetilde{\varphi}_r$ of $\varphi_r^{(4)}$ to $C^r_{(p,q)}$  also has a unique fixed point:  $(x_1,x_4)=(\lambda^{\frac{1}{4-r^2}},\lambda^{\frac{1}{4-r^2}})$. Each of these fixed points is a periodic point of $\varphi_r$ with minimal period 4. The full set of these 4-periodic points is  a (punctured) 2-dimensional variety $V\backslash\{F\}$.  Like in the proof of the previous theorem,  the explicit form of $V$ is easily obtained from the fact that the fixed point  of $\widetilde{\varphi}_r$ satisfies $x_1=x_4= \lambda^{\frac{1}{4-r^2}}$ and from the fact that $\lambda$ is the value of the restriction to $C^r_{(p,q)}$ of the function $l(\mathbf{x})$ given in \eqref{Flambda}. This completes the proof of statements  {\it 1.} and {\it 2.}

As $\bar\varphi_r$ and $\widetilde\varphi_r$ do not have more periodic points, any other point $\mathbf{x}$ is a non-periodic point of $\varphi_r$. 

If $\mathbf{x}\in C^r_{(1,1)}$ the $\varphi_r$-orbit of $\mathbf{x}$ remains in $C^r_{(1,1)}$. As $\bar\varphi_r$ is conjugate to $f_{r}$ by $\bar\pi_r(x_1,x_2)= 2^{\frac{1}{r-2}}(x_1,x_2)$  (cf.  Corollary~\ref{cor1}), and $f_r$ admits the first integral $I_r$ given by \eqref{int},  then $I_r\circ \bar\pi_r$ is a first integral of $\bar\varphi_r$. The  first integral $J$ in statement {\it 3.a)} is,  up to constants, $I_r\circ \bar\pi_r$. Hence,  the orbit of any point $\mathbf{x}$ belonging to $C^r_{(1,1)}$ lies in the curve 
$$L_\mathbf{x}=\{{\bf z}\in C^r_{(1,1)}: J({\bf z})=J(\mathbf{x})\}.$$
 The remaining assertions in statement {\it 3.a)} follow  from  Proposition~\ref{propbar}-{\it iii)} and  from the particular form of  the  map $\varphi_r$.  Indeed, as the third and fourth coordinates of $\varphi_r^{(n)}$ are respectively the first and second coordinates of $\varphi_r^{(n-1)}$, it follows that, if the first and second coordinates of $\varphi_r^{(n)}$ tend to zero (or to $+\infty$), the same holds for the third and fourth coordinates. 

If $\mathbf{x} \notin C_{(1,1)}^r$,  then by Proposition~\ref{orbits} its $\varphi_r$-orbit moves cyclically between the following four pairwise disjoint algebraic varieties of dimension 2
$$C^r_{\pi_r(\mathbf{x})}, \quad C^r_{\pi_r(\varphi_r(\mathbf{x}))}, \quad C^r_{\pi_r(\varphi_r^{(2)}(\mathbf{x}))}, \quad C^r_{\pi_r(\varphi_r^{(3)}(\mathbf{x}))}.$$

Let   $\widetilde\varphi_r^i$ denote the restriction of $\varphi_r^{(4)}$ to $C^r_{\pi_r(\varphi_r^{(i)}(\mathbf{x}))}$, for $i=0,\ldots ,3$. By Corollary~\ref{cor1}-{\it 2.ii)}, each map $\widetilde\varphi_r^i$ is conjugate to the map $f_{(r^2-2)^2-2}$ in \eqref{eq:tildefiproj}
by the conjugacy $\widetilde \pi_r^i$ given by
$$\widetilde \pi_r^i(x,y)= \lambda_i^{\frac{1}{r^2-4}} \left(\lambda_i \frac{x^{r^2-2}}{y}, y\right), $$
where $\lambda_i$ is defined by the expression in \eqref{lambda} with $(p,q)$  replaced by  $(p_i,q_i)=\pi_r (\varphi^{(i)}(\mathbf{x}))$.

As the map $f_{(r^2-2)^2-2}$ admits the first integral $I_{(r^2-2)^2-2}$ in \eqref{int}, then $I_{(r^2-2)^2-2} \circ \widetilde \pi_r^i$ is a first integral for $\widetilde\varphi_r^i$. Discarding multiplicative and additive constants in the computation of $I_{(r^2-2)^2-2} \circ \widetilde \pi_r^i$ and noting again that each parameter $\lambda_i$ is the restriction to $C^r_{\pi_r(\varphi_r^{(i)}(\mathbf{x}))}$ of the function $l$ in \eqref{Flambda} we obtain the expression $J$ in statement {\it 3.b)}. 

Finally note that for $i=0,\ldots,3$   the curve ${\cal L}_{\varphi_r^{(i)}(\bf{x})}$ lies on the 2-dimensional algebraic variety $C^r_{\pi_r(\varphi_r^{(i)}(\mathbf{x}))}$, and so the results in Proposition~\ref{proptilde}-{\it iii)} can be restated for the restriction of $\varphi_r^{(4)}$ to $C^r_{\pi_r(\varphi_r^{(i)}(\mathbf{x}))}$ as follows: either $\varphi_r^{(4n+i)}(\mathbf{x})$ converges to a 4-periodic point in $V$ or it converges to $(0,0,0,0)$ or its four components go to $+\infty$.  Using the expression of the map $\varphi_r$, we can then conclude that $\varphi_r (\varphi_r^{(4n+i)}(\mathbf{x}))$ has precisely the same behaviour, that is it converges to a 4-periodic point in $V$ if that was the case with $\varphi_r^{(4n+i)}(\mathbf{x})$, or converges to $(0,0,0,0)$ if that was the the case with $\varphi_r^{(4n+i)}(\mathbf{x})$, or all its components go to $+\infty$ if that was the case with $\varphi_r^{(4n+i)}(\mathbf{x})$. 
\end{proof}

\section{Conclusions and discussion}

The dynamics of cluster maps associated to mutation-periodic quivers is a recent subject of research.  In \cite{FoHo2} and \cite{FoHo1} we can find the study of  integrability and algebraic entropy of cluster maps arising from 1-periodic quivers. In \cite{InHeEs} we have addressed the study of the dynamics of cluster maps arising from the 2-periodic quivers associated to the Hirzebruch 0 and del Pezzo 3 surfaces. 

In the present work we completely described the dynamics of a parameter-dependent family of cluster maps in dimension 4 which arise from mutation-periodic quivers of period 2. This successful description relies on: (a) the possibility of reducing the original maps to symplectic maps of the plane (see \cite{InEs2}); (b) the global periodicity of the symplectic reduced maps, which allows us to study the dynamics of the original family by restricting  appropriate maps to  2-dimensional  symplectic varieties; (c) the fact that these restricted maps  belong to a group $\Gamma$ of birational (symplectic) maps of the plane isomorphic to $SL(2,\Zb)\ltimes \Rb^2$. Parametrizing $\Gamma$ by the trace of a $SL(2,\Zb)$ matrix and a real parameter (cf. Theorem~\ref{FormasNormais}) we obtained two classes of maps (normal forms) whose dynamics we described in detail. As a consequence we recovered the dynamics of the restricted maps (cf. propositions~\ref{propbar} and \ref{proptilde}) and then that of the original maps (cf. theorems~\ref{tr=1} and \ref{tr>2}).

The techniques we have used apply to any cluster map defined on an $N$-dimensional domain for which the reduced symplectic map is a globally periodic map in dimension $2k$. However,  the computations will become more involved as $N$ or the global period of the symplectic map increases. Also, when $N-2k$ is greater than 2, the respective restricted maps will belong to a group of symplectic maps with a more complicated structure than that of $\Gamma$. 

A  cluster map is  a birational map which, in general, has a  complicated expression due to the nature of the definition of a mutation. The existence of a (semi)conjugacy between a cluster map and a globally periodic map is not an obvious property of the cluster map, and there is no general procedure to prove its existence. In the case of the family treated here the presymplectic reduction enabled us to find an  explicit semiconjugacy to a globally periodic map. The study of the integrability either of the symplectic reduced map or of the cluster map may shed light on the existence of such a semiconjugacy.

The search for first integrals of a cluster map is usually linked to the existence of Poisson structures preserved by the cluster map. For instance, such Poisson structures exist  for a 3-parameter family of cluster maps associated to 2-periodic quivers of 4 nodes which include the family  $\varphi_r$  as a particular case  (see \cite[ \textsection 2, \textsection 5]{InEs2} for more details). The study of these Poisson structures and  first integrals arising from them, may provide  answers to the difficult question of knowing which mutation-periodic quivers give rise to cluster maps that are (semi)conjugate to globally periodic maps and consequently to the description of their dynamics.

\medskip

\noindent {\bf Acknowledgements}.   

The work of I. Cruz and H. Mena-Matos was partially funded by FCT under the project PEst-C/MAT/UI0144/2013.

The work of  M. E. Sousa-Dias  was partially funded by FCT/Portugal through the projects  UID/MAT/04459/2013 and EXCL/MAT-GEO/0222/2012.

\small{

\end{document}